\documentclass[12pt,a4paper]{amsart}

\usepackage[foot]{amsaddr}

\usepackage{amsmath,amssymb,amsthm,graphicx,cite,color}

\calclayout

\DeclareMathOperator*{\aff}{aff}
\DeclareMathOperator*{\cl}{cl}
\DeclareMathOperator*{\co}{co}
\DeclareMathOperator*{\cone}{cone}
\DeclareMathOperator*{\dist}{dist}
\DeclareMathOperator*{\interior}{int}
\DeclareMathOperator*{\lspan}{span}
\DeclareMathOperator*{\relint}{relint}
\DeclareMathOperator*{\lin}{lin}

\newcommand{\R}{\mathbb{R}}

\newtheorem{theorem}{Theorem}[section]

\newtheorem{corollary}[theorem]{Corollary}
\newtheorem{lemma}[theorem]{Lemma}
\newtheorem{proposition}[theorem]{Proposition}

\theoremstyle{definition}

\newtheorem{definition}[theorem]{Definition}
\newtheorem{example}[theorem]{Example}
\newtheorem{remark}[theorem]{Remark}
\newtheorem{question}[theorem]{Question}

\title{Dimension and structure of higher-order Voronoi cells on discrete sites}

\author{Ryan McKewen}
\author{Vera Roshchina}
\address[Ryan McKewen,Vera Roshchina]{School of Mathematics and Statistics, UNSW Sydney}
\email[Ryan McKewen]{r.mckewen@student.unsw.edu.au}
\email[Vera Roshchina]{v.roshchina@unsw.edu.au}



\dedicatory{}

\begin{document}
	
\maketitle	
	
\begin{abstract} We study the structure of higher-order Voronoi cells on a discrete set of sites in $\R^n$, focussing on the relations between cells of different order, and paying special attention to the ill-posed case when a large number of points lie on a sphere. In particular, we prove that higher order cells of dimension $n-1$ do not exist, even though high-order Voronoi cells may have empty interior. We also present a number of open questions.
\end{abstract}

\section{Introduction}

Voronoi cells are used in computer graphics, crystallography, facility location, and have numerous other applications. The first recorded use of a Voronoi cell-like object goes back to Ren\'e Descartes \cite{descartes}, while mathematical foundations were initially developed by Dirichlet \cite{dirichlet}, Voronoi \cite{voronoi} and Delone \cite{delone}. Various generalisations of Voronoi cells are of significant practical importance and provide a source of curious mathematical problems.

A classic Voronoi diagram is a tessellation of the Euclidean space by cells generated from  a discrete set of points called sites: each cell consists of points which are no farther from a given site than from the remaining ones. A higher-order diagram is  a generalisation of this notion, where points nearest to several sites are considered. To our best knowledge, the latter notion first appeared in \cite{shamos}.

There are many applications of higher-order diagrams in diverse fields. Some recent examples include mobile sensor coverage control problems \cite{mobile,hole},  $k$-nearest
neighbour problems in spatial networks \cite{spatial}, smoothing point clouds in higher dimensions \cite{gradflow}, texture generation in computer graphics \cite{texture}, detection of symmetries in discrete point sets \cite{paradoxes}, analysis and modelling of voting in US Supreme Court \cite{supreme}. Higher-order cells are also used as a tool in resolving mathematical problems, for instance, see \cite{edelsbrunner,circlesenclosing} where Ramsey style bounds are obtained on the maximal number of points enclosed by circles passing through a pair of coloured points.

An extreme special case of higher-order diagrams is the farthest Voronoi cells studied in \cite{farthest}. In particular, this setting motivates the generalisation of the notion of boundedly exposed points used by the authors to obtain a characterisation of nonempty farthest Voronoi cell. Our work is in a similar spirit: we are focussed on structural properties of higher-order Voronoi cells, motivated by the puzzling observations on low-dimensional cells presented in \cite{multipoint} and by the work \cite{ryan} focussed on generating regular tessellations from higher-order cells. 

Our main contribution is a characterisation of the dimension of a higher-order Voronoi cell in Theorem~\ref{thm:dimension}, and a somewhat unexpected corollary that higher-order cells in $\R^n$ can not have dimension $n-1$ (while lower dimensions are possible). We also discuss some relations between cells of different order.

Our paper is organised as follows. We remind the definition and basic facts about higher-order Voronoi diagrams in Section~\ref{sec:structure}, and follow with a characterisation of the dimension of a higher-order Voronoi cell in Theorem~\ref{thm:dimension} of Section~\ref{sec:dimension}. In Section~\ref{sec:order} we discuss the relations between cells of different orders, providing some illustrative examples, and in Section~\ref{sec:neighbour} address issues pertaining to neighbour relations, specifically focussing on degenerate cases. We finish with a brief discussion of our results in  Section~\ref{sec:conclusions}, where we also present some open questions.

\section{Basic structural properties of higher-order Voronoi cells}\label{sec:structure}

Let $S,T\subseteq \R^n$. We define a Voronoi cell of $S$ with respect to $T$ as 
\begin{equation}\label{eq:defVoronoi}
V_{T}(S):=\left\{  x\in\mathbb{R}^{n}\,:\,\sup_{s\in S}\dist (x,s)\leq
\inf_{t\in T\setminus S}\dist (x,t)\right\},
\end{equation}
where $\dist(x,s) = \|x-s\|$ is the Euclidean distance in $\R^n$. Whenever the set $T\setminus S$ is nonempty and closed, and $S$ is compact, the infimum and supremum can be replaced by the minimum and maximum, respectively. When $S$ is finite, the \emph{order} of the cell $V_T(S)$ is its cardinality $|S|$.

The following consequence of the definition \eqref{eq:defVoronoi}  is key to several proofs in this paper. 

\begin{theorem}[{cf. \cite[Theorem~2.2]{multipoint}}]\label{thm:ballchar} Let $S,T\subseteq \R^n$. Then $V_T(S)\neq \emptyset$ if and only if there exists a closed Euclidean ball $B$  such that 
\begin{equation}\label{eq:ballchar}
S\subseteq B, \qquad \interior B \cap (T\setminus S) = \emptyset.
\end{equation}
Moreover,  $x\in V_T(S)$ if and only if there exists a Euclidean ball $B$ centred at $x$ that satisfies \eqref{eq:ballchar}. 
\end{theorem}
\begin{proof}
There exists a Euclidean ball of radius $r$ centred at $x$ satisfying \eqref{eq:ballchar} if and only if
\[
\|s-x\|\leq r\quad \forall s\in S, \quad \|t-x\|\geq r \quad \forall t \in T\setminus S;
\]
equivalently $\|s-x\|\leq  \|t-x\|$ for all $s\in S$ and all $t\in T\setminus S$, i.e. $x\in V_T(S)$.
\end{proof}

It may happen that for a point in a higher-order Voronoi cell there are several choices of a Euclidean balls that satisfy \eqref{eq:ballchar}. It is intriguing to explore the relation between the flexibility of this choice and the degeneracy and singularity in the cell structure. We address this in more detail in Section~\ref{sec:conclusions}.

To illustrate the usefulness of Theorem~\ref{thm:ballchar}, we show that the characterisation of a nonempty cell obtained in  \cite[Theorem~22]{farthest} follows directly from this result.

\begin{corollary}Let $T\subseteq\R^n$, $s\in \R^n$. The farthest cell $V_{\{s\}}(T\setminus \{s\})$ is nonempty if and only if $T$ is bounded and there exists a closed Euclidean ball $B$ such that
\begin{equation}\label{eq:bdreg}
T\setminus \{s\}\subset \interior B, \quad s\in \partial B.
\end{equation}
\end{corollary}
\begin{proof} It follows from Theorem~\ref{thm:ballchar} that for $V_{\{s\}}(T\setminus \{s\})\neq \emptyset$ it is necessary and sufficient to have a Euclidean ball $B'$ such that 
\begin{equation}\label{eq:modbd}
T\setminus \{s\}\subseteq B', \qquad \interior B' \cap \{s\} = \emptyset.
\end{equation}
If \eqref{eq:bdreg} holds for some ball $B'$, it yields  \eqref{eq:modbd}. Conversely, assume that \eqref{eq:modbd} holds, let $x$ be the centre of the ball $B'$, and let $r$ be its radius. It is evident that $r\leq \|x-s\|$, moreover, 
whenever $s-x$ and $t-x$ are non-collinear, we have
\[
\|t-(2x-s)\| < \|t-x\|+\|s-x\| = 2 \|s-x\|.
\]
When for some $t\in T$ the vectors $s-x$ and $t-x$ are collinear and $t\neq s$, we have $t-x = k (s-x)$ with $-1\leq k<1$.  Hence, 
\[
\|t-(2x-s)\| = \|(t-x) + (s-x)\| = (1+k)\|s-x\| <2 \|s-x\|.
\]
We deduce that 
\[
\|t-(2x-s)\| <2 \|s-x\| \quad \forall t\in T\setminus \{s\}.
\]
Together with 
\[
\|s-(2x -s)\|= 2 \|x-s\|
\]
this means that the Euclidean ball $B$ of radius $2\|s-x\|$ centred at $2x-s$ satisfies \eqref{eq:bdreg}.
\end{proof}

The point $s$ satisfying property \eqref{eq:bdreg} is called boundedly exposed extreme point of the set $\cl \co (T\cup \{s\})$ (see \cite{farthest} for a detailed explanation and a generalisation of the original notion introduced in \cite{Edelstein}).

The following elementary result will be useful for subsequent proofs.
\begin{proposition}[{\cite[Proposition~2.1]{multipoint}}]\label{prop:handycharacterisation} Let $T$ and $S$ be subsets of $\R^{n}$. Then $V_{T}\left(  S\right)  $ can be represented as the intersection of closed halfspaces,%
\begin{equation}
V_{T}(S)=\bigcap_{\substack{s\in S \\t\in T\setminus S}}\left\{
x\in\R^{n}\,:\,\langle t-s,x\rangle\leq\frac{1}{2}\left(  \Vert
t\Vert^{2}-\Vert s\Vert^{2}\right)  \right\}  .\label{eq:linrep}%
\end{equation}

\end{proposition}

\section{Dimensions of higher-order Voronoi cells}\label{sec:dimension}

We refine Theorem~\ref{thm:ballchar} in the following explicit characterisation of   dimensions of high-order cells.

\begin{theorem}\label{thm:dimension} Let $S,T\subset \R^n$ be such that $T$ is discrete and $S$ is finite. Suppose that $B$ is a closed Euclidean ball such that $S\subseteq B$ and $(T\setminus S )\cap \interior B = \emptyset$. Let 
\begin{equation}\label{eq:intersectionC}
C := \co (S\cap \partial B)\cap \co ((T\setminus S)\cap \partial B),
\end{equation}
and let $F_S$ and $F_T$ be the minimal faces of $\co (S\cap \partial B)$ and $\co ((T\setminus S)\cap \partial B)$ respectively that contain $C$. 
Then 
$$
\dim V_T(S) = \begin{cases}
n\, & \text{if } C = \emptyset,\\
n- \dim \co \{F_S,F_T\}, & \text{if } C \neq \emptyset.
\end{cases}
$$
\end{theorem}

\begin{figure}[ht]
\includegraphics[width = 0.8\textwidth]{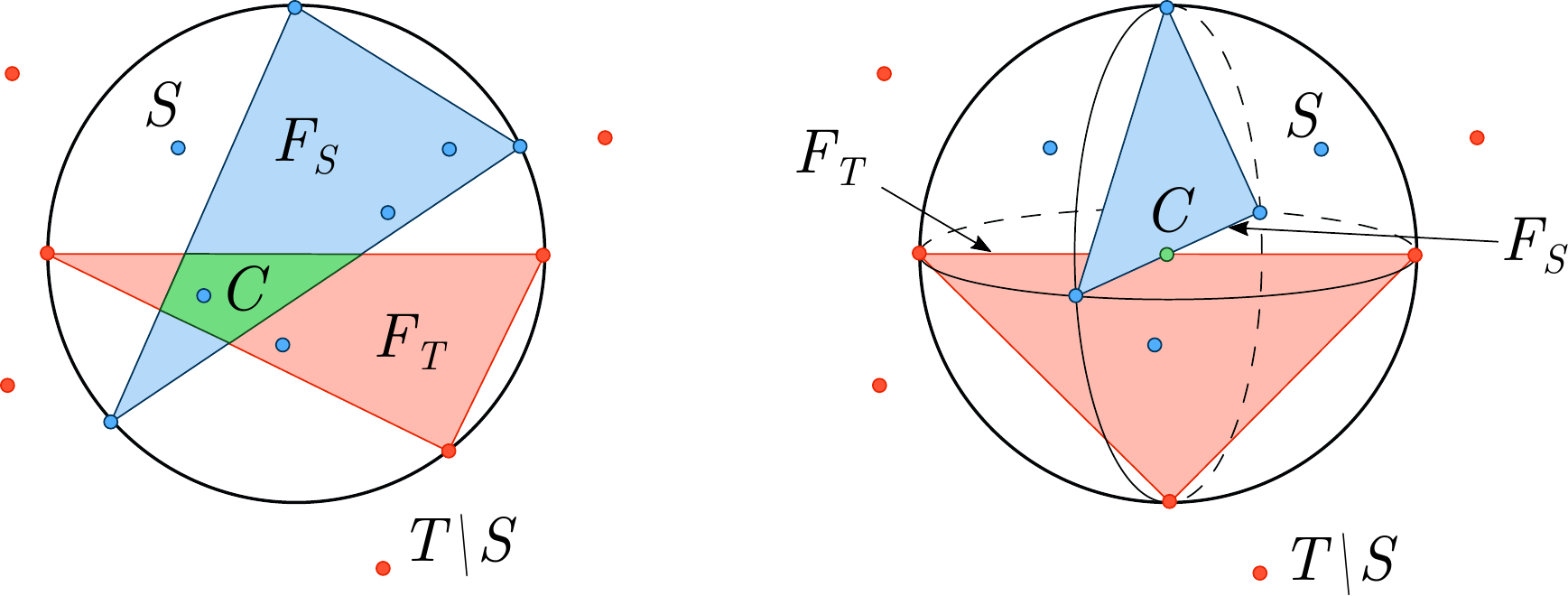}
\label{fig:illustration}
\caption{Illustration for Theorem~\ref{thm:dimension} }
\end{figure}

We will use the following two results in the proof of Theorem~\ref{thm:dimension}.

\begin{proposition}\label{prop:dim} Let $C$ and $D$ be convex sets in $\R^n$ such that $C\cap \interior D\neq \emptyset$. Then $\dim C = \dim (C\cap D)$.
\end{proposition}
\begin{proof} Since $C\cap D \subseteq C$, we have $\dim C \geq \dim C\cap D$. Suppose that $\dim C\cap D<\dim C$. Then there exists $y\in C\setminus \aff (C\cap D)$. Let $x\in C\cap \interior D$. For the line segment $[x,y]$ connecting $x$ and $y$ we have $[x,y]\subset C$, also since $x\in \interior D$ there exists a sufficiently small $\alpha\in (0,1]$ such that $x+ \alpha (y-x)\in D$. Therefore, $x+\alpha (y-x) \in (C \cap D)\subset \aff (C \cap D) $. Since $x\in \aff (C \cap D)$ and $x+\alpha (y-x) \in \aff (C \cap D)$, we must also have $x+ (y-x) = y\in \aff (C \cap D)$, which contradicts the assumption.
\end{proof}

\begin{proposition}\label{prop:dimdual} Let $K\subseteq\R^n$ be a closed convex cone and let 
$$
K^\circ = \{y\in \R^n \, |\, \langle x,y\rangle \leq 0 \; \forall x\in K\} 
$$
be its (negative) polar. Then $\dim K^\circ = n - \dim \lin K$.
\end{proposition}
\begin{proof}[Proof]
Fix any $x \in \lin K$. We have $\langle x,y\rangle =0$ for all $y \in K^\circ$, hence, $K^\circ \subseteq (\lin K)^\perp$, and $\dim K^\circ \leq n-\dim \lin K$.

Since $K = \lin K + K'$ with  $K'\subset (\lin K)^\perp$ pointed (see \cite[Lemma~5.33]{Guler}), within the space $(\lin K)^\perp$ the set $\co [(K \cap (\lin K)^\perp)\cap S]$, where $S$ is the unit sphere, can be strictly separated from zero (see \cite[Corollary 4.1.3]{FundamentalsConvexAnalysis}), which means that the dual cone to $K \cap (\lin K)^\perp$ within $(\lin K)^\perp$ has a nonempty interior. Hence we have $\dim K^\circ \geq \dim (\lin K)^\perp =   n-\dim \lin K$.
\end{proof}

\begin{proof}[Proof of Theorem~\ref{thm:dimension}] Without loss of generality assume that the Euclidean ball $B$ is centred at zero. Then by Proposition~\ref{prop:handycharacterisation} we have 
\begin{equation}\label{eq:ineqrepres}
V_T(S) = \bigcap_{\substack{s\in S\\t\in T\setminus S} }\left\{x\in \R^n\,\Bigl| \, \langle t-s,x\rangle \leq \frac{1}{2}(\|t\|^2 - \|s\|^2)\right\} = V \cap V' \cap V'',
\end{equation}
where 
\begin{align*}
V & = \bigcap_{\substack{s\in S\cap \partial B\\t\in (T\setminus S)\cap \partial B} }\left\{x\in \R^n\,\Bigl| \, \langle t-s,x\rangle \leq 0\right\} \quad (\text{since } \|t\|=\|s\|),\\
V'&  =   \bigcap_{\substack{s\in S\cap \interior B\\t\in (T\setminus S)\cap \partial B} }\left\{x\in \R^n\,\Bigl| \, \langle t-s,x\rangle \leq \frac{1}{2}(\|t\|^2 - \|s\|^2)\right\},\\
V'' & = \bigcap_{\substack{s\in S\\t\in T\setminus B }}\left\{x\in \R^n\,\Bigl| \, \langle t-s,x\rangle \leq \frac{1}{2}(\|t\|^2 - \|s\|^2)\right\}.
\end{align*}	

We have $\|t\|>\|s\|$ for all pairs $(s,t) \in \interior B \times (\R^n\setminus \interior B)$ and $(s,t) \in B \times (\R^n \setminus B)$. Since $T$ is discrete and $S$ is finite, there exists a sufficiently small $\varepsilon>0$ such that for a ball $B_\varepsilon$ of radius $\varepsilon$ centred at zero we have $B_\varepsilon \subset V'\cap V''$. Hence, we have 
\begin{equation}\label{eq:intballs}
V_T(S)\cap B_\varepsilon = (V\cap B_\varepsilon)\cap (V'\cap B_\varepsilon) \cap (V''\cap B_\varepsilon) = V\cap B_\varepsilon.
\end{equation}
Observe that for any $s\in S$ and $t\in T\setminus S$ we have $\|t\|\geq \|s\|$, hence, it is clear from \eqref{eq:ineqrepres} that $0\in V_T(S)\subseteq V$, so $V_T(S) \cap \interior B_\varepsilon \neq 0 $. We hence have from \eqref{eq:intballs} and Proposition~\ref{prop:dim}
$$
\dim V_T(S) = \dim V_T(S)\cap B_\varepsilon = \dim V \cap B_\varepsilon = \dim V.
$$
Observe that $V = (\cone [((T\setminus S)\cap \partial B)-(S\cap \partial B)])^\circ$. By Proposition~\ref{prop:dimdual} we have 
$$
\dim V = n- \dim \lin \cone [((T\setminus S)\cap \partial B)-(S\cap \partial B)].
$$
It remains to show that 
$$
\dim \lin \cone [((T\setminus S)\cap \partial B)-(S\cap \partial B)] = \dim \co \{F_S, F_T\},
$$
from which the result would follow. We break this down into two steps: first we show that
\begin{equation}\label{eq:4356546745}
\lin \cone [((T\setminus S)\cap \partial B)-(S\cap \partial B)] = \cone \{F_S- F_T\},
\end{equation}
and then demonstrate that 
\begin{equation}\label{eq:dimequal}
\dim \cone \{F_S- F_T\} = \dim\co \{F_T,F_S\}.
\end{equation}

Observe that $z\in \lin \cone [((T\setminus S)\cap \partial B)-(S\cap \partial B)] $ if and only if for some $w = \alpha z$ we have $-w, w\in \co ((T\setminus S)\cap \partial B)-\co (S\cap \partial B) $. Then 
$$
w = u_1-v_1, \quad -w = u_2-v_2,\qquad u_1,u_2\in \co ((T\setminus S)\cap \partial B),\quad  v_1,v_2\in \co (S\cap \partial B), 
$$
hence, $u_1+u_2 = v_1+v_2$, and effectively 
$$
(u_1,u_2)\cap (v_1,v_2) \neq \emptyset,
$$
which implies that $u_1,u_2$ and $v_1,v_2$ belong to the minimal faces $F_S$ and $F_T$ containing $C$. Therefore, 
$$
\lin \cone [((T\setminus S)\cap \partial B)-(S\cap \partial B)] \subseteq \cone \{F_S - F_T\}.
$$ 
Now suppose that $z\in \cone \{F_T-F_S\}$. We will show that $-z\in \cone \{F_T-F_S\}$, hence,
\begin{equation}\label{eq:8745345}
\cone \{F_T-F_S\} = \lin \cone\{F_T-F_S\} \subseteq \lin \cone [((T\setminus S)\cap \partial B)-(S\cap \partial B)].
\end{equation}

We have 
\begin{equation}\label{eq:zexp}
z = \sum_{i=1}^m \lambda_i (p_i-q_i),
\end{equation}
where $p_i\in F_T$, $q_i \in F_S$, and $\lambda_i\geq 0$ for all $i = 1,2,\dots, m$.

Now since $F_T$ and $F_S$ are the minimal faces, we must have $\relint C \subseteq \relint F_T\cap \relint F_S$, and hence there exists $c\in \relint F_T\cap \relint F_S$. We can then find an $\alpha>0$ such that 
\[
c+ \alpha (c-p_i)=: p_i'\in F_T, \quad c+ \alpha(c-q_i)=: q_i'\in F_S\qquad \forall i = 1,2,\dots, m.
\] 
Rearranging, we obtain
\[
p_i = \frac{1+\alpha}{\alpha} c - \frac{1}{\alpha}p'_i, \quad 
q_i = \frac{1+\alpha}{\alpha} c - \frac{1}{\alpha}q'_i \qquad \forall \, i = 1,2,\dots, m.
\]
Substituting this into \eqref{eq:zexp}, we have 
\[
z = \sum_{i=1}^m \lambda_i \left(\left[\frac{1+\alpha}{\alpha} c - \frac{1}{\alpha}p'_i\right]-\left[\frac{1+\alpha}{\alpha} c - \frac{1}{\alpha}q'_i\right]\right) =- \sum_{i=1}^m \frac{\lambda_i}{\alpha} \left(p_i'-q_i'\right) \in  - \cone \{F_T-F_S\} ,
\]
hence, $-z \in \cone \{F_T-F_S\} $ and we are done with \eqref{eq:8745345}, which finishes the proof of \eqref{eq:4356546745}.

It remains to show \eqref{eq:dimequal}. Observe that for any $c\in \R^n$ we have 
\[
\dim \co \{F_T,F_S\} = \dim \co \{F_T-\{c\},F_S-\{c\}\},
\]
hence it suffices to show that 
\[
\aff \co \{F_T-\{c\},F_S-\{c\}\} = \aff  \cone \{F_T-F_S\}.
\]
Let $c\in \relint F_T\cap \relint F_S$, as before. Since $0\in F_T-\{c\}$ and $0\in F_T-F_S$, we have
\[
 \aff \co \{F_T-\{c\},F_S-\{c\}\} =  \lspan \co \{F_T-\{c\},F_S-\{c\}\} =   \lspan \cone\{(F_T-\{c\})\cup(F_S-\{c\})\},
\]
and 
\[
\aff \cone \{F_T-F_S\} = \lspan \cone \{F_T-F_S\}.
\]
Therefore, it is sufficient to show that $\cone \{F_T-F_S\} = \cone[ (F_T-\{c\})\cup(F_S-\{c\})]$.

Let $z\in \cone \{F_T-F_S\}$. Then 
\begin{equation}\label{eq:repr3}
z = \sum_{i=1}^m \lambda_i (p_i-q_i), \quad \lambda_i \geq 0, \; p_i \in F_T, \; q_i \in F_S.
\end{equation}
Observe that since $c\in \relint F_S$, for every $i$ there exists $\alpha_i>0$ such that 
\[
c+ \alpha_i (c- q_i) =: q_i'\in F_S,
\]
hence $c-q_i = \frac{1}{\alpha_i} (q_i'-c)$. From \eqref{eq:repr3} we have 
\[
z = \sum_{i=1}^m \lambda_i [(p_i-c)-(q_i-c)] = \sum_{i=1}^m \lambda_i (p_i-c)+\sum_{i=1}^m \frac{\lambda_i}{\alpha_i}(q_i'-c), 
\]
hence, $z\in \cone[(F_T-\{c\})\cup(F_S-\{c\})]$.

Now assume $z\in \cone[(F_T-\{c\})\cup(F_S-\{c\})]$. Then 
\[
z = \sum_{i=1}^m \lambda_i (p_i -c)+\sum_{j=1}^k \mu_j (q_i -c), \quad \lambda_i,\mu_j\geq 0, p_i \in F_T, q_i\in F_S.
\]
If $\sum_{i=1}^m\lambda_i = 0$, then every $\lambda_i=0$, and $z\in \cone\{F_S-\{c\}\}\subseteq \cone \{F_S-F_T\} = \cone \{F_T-F_S\}$ (see the preceding discussion on lineality spaces).

If $\sum_{j=1}^k \mu_j = 0$, then every $\mu_j = 0$ and $z\in \cone\{F_T-\{c\}\}\subseteq \cone \{F_T-F_S\} $.

Finally if both sums are positive, we have 
\begin{align*}
z & = \sum_{i=1}^m \lambda_i \cdot \frac{\sum_{j=1}^k \mu_j}{\sum_{j=1}^k \mu_j}  (p_i -c)+\sum_{j=1}^k \mu_j \cdot \frac{\sum_{j=1}^m \lambda_i}{\sum_{j=1}^m \lambda_i} (q_i -c)\\ 
& = \sum_{i=1}^m \lambda_i \cdot \frac{\sum_{j=1}^k \mu_j}{\sum_{j=1}^k \mu_j}  (p_i -c)+\sum_{j=1}^k \mu_j \cdot \frac{\sum_{j=1}^m \lambda_i}{\sum_{j=1}^m \lambda_i} (q_i -c)\\
& = \sum_{i=1}^m \sum_{j=1}^k \lambda_i \mu_j\left(  \frac{1}{\sum_{j=1}^k \mu_j}  (p_i -c)+ \frac{1}{\sum_{j=1}^m \lambda_i} (q_i -c)\right).
\end{align*}
Observe that since $c\in \relint F_T$ and $c\in \relint F_S$, there exists a sufficiently small $\alpha>0$ such that 
\[
c +\frac{\alpha}{\sum_{j=1}^k \mu_j}  (p_i -c) = p'_i \in F_T, \quad 
c -\frac{\alpha}{\sum_{i=1}^m \lambda_i}  (q_j -c) = q'_j \in F_S, 
\]
so we have 
\[
z = \sum_{i=1}^m \sum_{j=1}^k \lambda_i \mu_j\left(  \frac{1}{\alpha}  (p'_i -c)+ \frac{1}{\alpha} (c-q_j')\right)= \sum_{i=1}^m \sum_{j=1}^k \frac{\lambda_i \mu_j}{\alpha}\left( p'_i - q'_i\right) \in \cone \{F_T-F_S\}.
\]

\end{proof}	

It would be curious to see if Theorem~\ref{thm:dimension} can be generalised to a non-discrete setting.
	
\begin{corollary}\label{cor:nodimn1} If $T\subset \R^n$ is discrete, there are no Voronoi cells of finite order that have dimension $n-1$ in $\R^n$.
\end{corollary}
\begin{proof} It follows from Theorem~\ref{thm:dimension} that for a higher-order Voronoi cell to have dimension  $n-1$ one needs to have a Euclidean ball satisfying the assumption of Theorem~\ref{thm:dimension} and  $\dim  \co \{F_S, F_T\}=1$. Since $\co \{F_S, F_T\}$ is the convex hull of some points from $(S\cap \partial B)\cup (T\setminus S \cap \partial B)$, it must be a line segment connecting two points, one of them from $S$, and the other one from $T\setminus S$. We conclude that the minimal faces $F_S$ and $F_T$ must be single points on the surface of the sphere, so   
$$
\emptyset = F_S\cap F_T = \co (S\cap \partial B)\cap \co (T\cap \partial B),
$$
and the Voronoi cell is empty, hence, it cannot have dimension $n-1$. 
\end{proof}

Note that Theorem~\ref{thm:dimension} and Corollary~\ref{cor:nodimn1} significantly generalise \cite[Proposition~3.1]{multipoint}. Observe that on the plane we can only have cells of dimensions 0 and 2. Furthermore, in the setting of $|T|=4$ and $|S|=2$ on the plane for a cell of dimension 0 we have $F_S\cap F_T \neq \emptyset$. The only possibility given the constraints on the cardinalities is to have these sets as two intersecting line segments with vertices on a circle. This is precisely the case of the diagonals of a cyclic quadrilateral.

\begin{example}[A cell of dimension 1 in $\R^3$] Let $S = \{(1,0,0),(-1,0,0),(0,0,1)\}$, $T\setminus S = \{(0,1,0),(0,-1,0), (0,0,-1)\}$. Observe that $\|t\|=1$ for every $t\in T$, and order 3 Voronoi cell $V_T(S)$ is defined by the system
\[
y-x\leq 0, \; y+x\leq 0,  \;-y-x\leq 0,  \;-y+x\leq 0, \; y-z\leq 0, \; -y-z\leq 0, \; -z\leq 0.
\]
From the first four inequalities we obtain $y\leq x\leq y$ and $-y\leq x\leq -y$, hence, $x=y=-y = 0$. Then the last three inequalities yield $z\geq 0$. Hence the Voronoi cell is exactly
$$
V_T(S) = 0_2 \times \R_+.
$$

We next verify Theorem~\ref{thm:dimension} for this example. Observe that the unit ball $B = \{(x,y,z)\,|\, x^2+y^2+z^2 \leq 1\}$ centred at 0 satisfies the conditions of the theorem, moreover, we have $S = S\cap \partial B$, $T = T\cap \partial B$.

For the intersection 
\begin{align*}
C & = \co(S\cap \partial B) \cap \co ((T\setminus S) \cap \partial B)\\
&  = \co \{(1,0,0),(-1,0,0), (0,0,1)\} \cap \co \{(0,1,0),(0,-1,0), (0,0,-1)\} 
\end{align*}
we have $(x,y,z)\in C$ iff
$$
x = \alpha_1-\alpha_2=0, y = \beta_1-\beta_2 = 0, z = \alpha_3 = - \beta_3, \quad \alpha_i,\beta_i \geq 0, \sum\limits_{i=1}^3 \alpha_i = \sum\limits_{i=1}^3 \beta_i =1. 
$$
We deduce that $\alpha_1=\alpha_2 = \beta_1 = \beta_2 = 1/2$, $\alpha_3 = \beta_3 = 0$, and $x=y=z=0$, so $C = \{0_3\}$.

It is evident that the relevant minimal faces are $F_S = \co\{(1,0,0),(-1,0,0)\}$, $F_T = \{(0,1,0),(0,-1,0)\}$, so $\dim \co \{F_S,F_T\} = 2$, and $\dim V_T(S) = 3-2=1$, which is consistent with the explicit expression for the cell that we obtained earlier. 
\end{example}

\section{Relations between cells of different orders}\label{sec:order}

The purpose of this section is to provide clear statements that generalise several well-known and fairly trivial properties of higher-order Voronoi cells. The next result shows that a higher-order Voronoi cell is the intersection of lower-order cells built from the set of defining sites. 

\begin{theorem}\label{thm:intersection} Let $T,S\subseteq \R^n$. Suppose that $\mathcal{U}$ is a family of subsets of $S$ such that 
\[
S = \bigcup_{U\in \mathcal{U}} U.
\] 
Then 
\[
V_T(S)  = \bigcap_{U\in \mathcal{U}} V_{T\setminus S}(U). 
\]
\end{theorem}
\begin{proof} Let $x\in V_T(S)$. Then 
\[
\|x-s\|\leq \|x-t\|\quad\forall s \in S, \; t\in T\setminus S.
\]
It follows that for every $U\in \mathcal{U}$ 
\[
\|x-s\|\leq \|x-t\|\quad\forall s \in  U, \; t\in T\setminus S.
\]
Therefore, 
\[
V_T(S)  \subseteq  \bigcap_{U\in \mathcal{U}} V_{T\setminus S}(U). 
\]

To show the converse, let 
\[
x\in  \bigcap_{U\in \mathcal{U}} V_{T\setminus U}(S\setminus U).
\]
Choose any $s\in S$. There exists $U\in \mathcal U$ such that $s\in U$, and therefore
\[
\|x-s\|\leq \|x-t\|\quad \forall t \in T\setminus S.
\]
Since $s\in S$ is arbitrary, we have by the definition of Voronoi cell $x\in V_T(S)$. 
\end{proof}

Let $T$ be a finite subset of $\R^n$, and let $S$ be a proper subset of $T$. Denote by $\mathcal{S}_k$ the set of all subsets of $S$ of cardinality exactly $k$, for all $k\in \{1,\dots, K\}$, where $K=|S|$. 

%It follows from Theorem~\ref{thm:intersection} that order $K$ Voronoi cell is the intersection of the Voronoi cells of any lower order, constructed on the reduced set of sites, as stated in the next corollary. 

\begin{corollary}\label{cor:intersection} Let $S,T\subset \R^n$ be such that $T$ is discrete and $S$ is finite.  Then for any $k\in \{1,\dots, |S|\}$
\[
V_T(S)  = \bigcap_{S'\in \mathcal{S}_k} V_{T\setminus S'}(S')= \bigcap_{S'\in \mathcal{S}_k} V_{T}(S'). 
\]
\end{corollary}

In particular, Corollary~\ref{cor:intersection} recovers the well-known result that order $k$ Voronoi cell is the intersection of first-order cells of all points in $S$ on the set of sites $T\setminus S$.

The next lemma is similar in spirit, providing an uppper estimate on the higher-order Voronoi cells from cells of lower order. However notice that in this case the lower cells are constructed on exactly the same set $T$. This result is helpful in identifying neighbouring points from the Voronoi diagram of a lower order. The essential difference between Theorem~\ref{thm:intersection} and Lemma~\ref{lem:inclusions} is illustrated in Fig.~\ref{fig:twoways}.

\begin{figure}[ht]
{\centering
\includegraphics[width=0.4\textwidth]{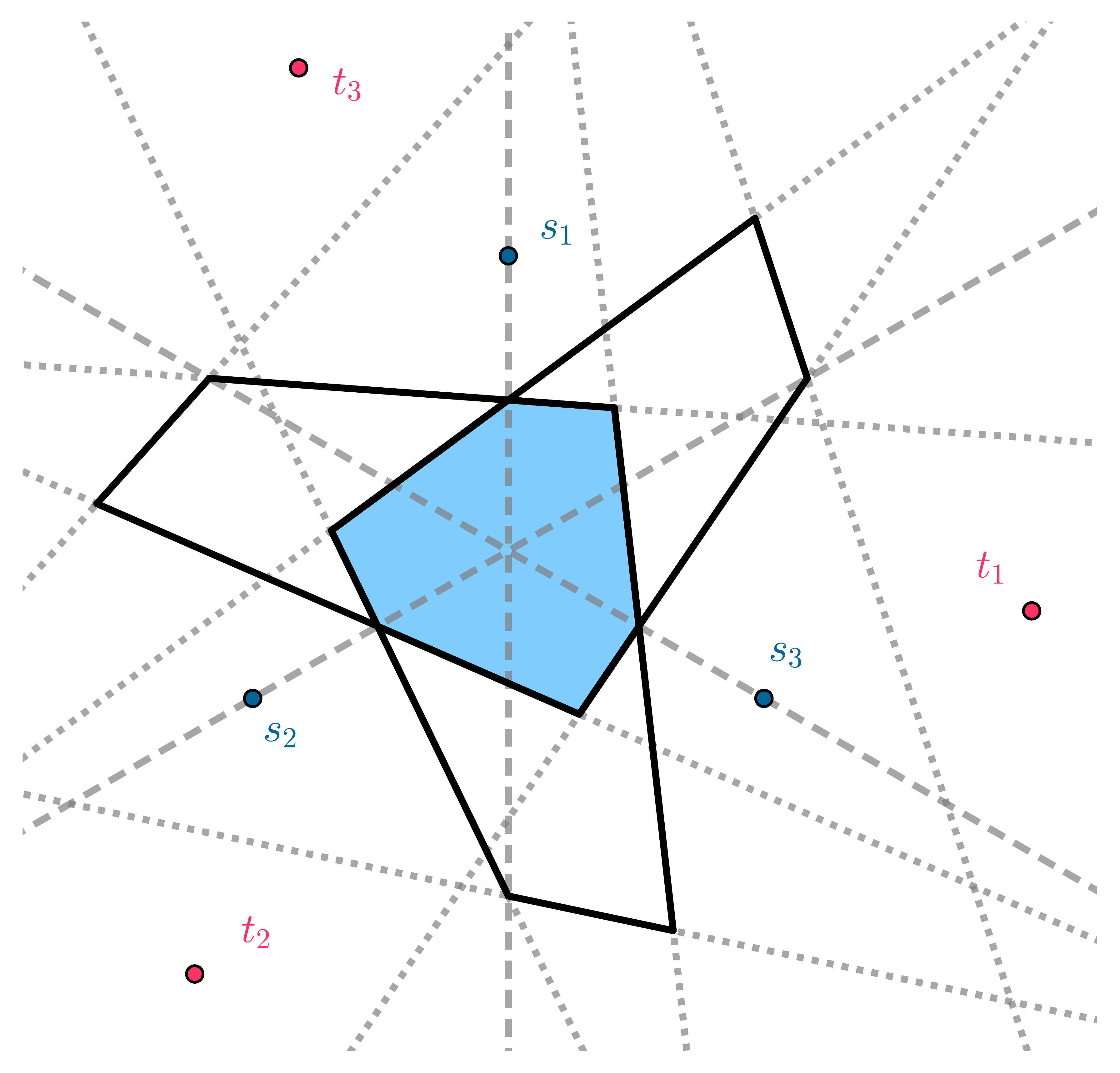}
\quad
\includegraphics[width=0.4\textwidth]{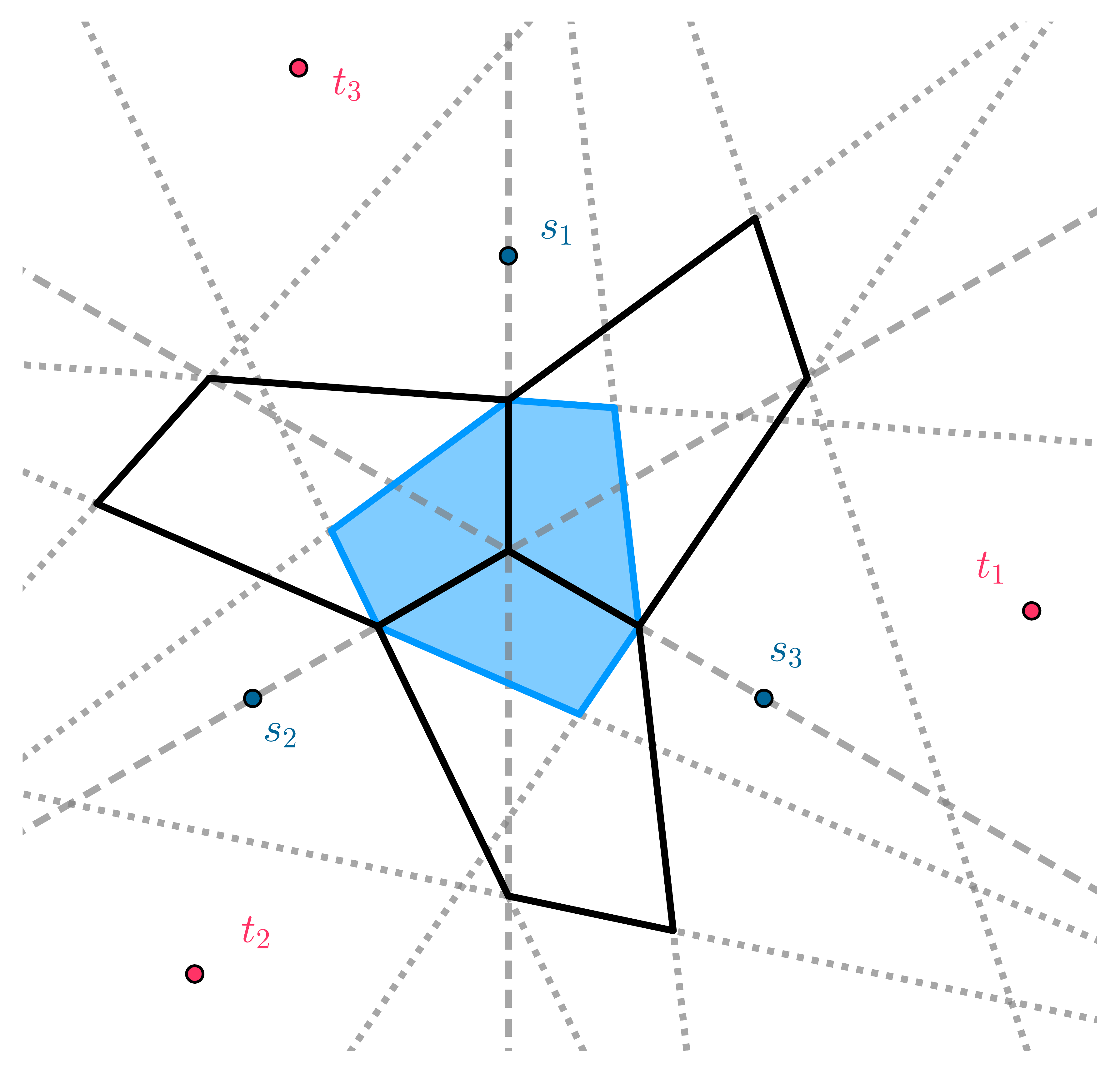}}
\caption{The representation of an order 3 Voronoi cell as the intersection of order 2 cells on the reduced set of sites, illustrating Theorem~\ref{thm:intersection} and an upper bound on the order 3 cell from the intersection of order 2 cells on the complete set of sites, illustrating  Lemma~\ref{lem:inclusions}.}
\label{fig:twoways}
\end{figure}

\begin{lemma}\label{lem:inclusions} Let $S,T\subseteq \R^n$, and assume that  $S\subset T$. As before, denote by $\mathcal{S}_k$ the set of all subsets of $S$ of cardinality exactly $k$, for all $k\in \{1,\dots, K\}$, where $K=|S|$. Then
\begin{equation}\label{eq:inclusions}
V_T(S)  \subseteq \bigcup_{S'\in \mathcal{S}_{K-1}}V_{T}(S')   \subseteq \bigcup_{S'\in \mathcal{S}_{K-2}} V_{T}(S')  \subseteq \cdots \subseteq   \bigcup_{S'\in \mathcal{S}_{1}} V_T(S').
\end{equation}
\end{lemma}
\begin{proof} Note that to prove the inclusions in \eqref{eq:inclusions} it is sufficient to show that
\[
V_T(S) \subseteq \bigcup_{S'\in \mathcal{S}_{K-1}}V_{T}(S'), 
\] 
the rest follows by induction. 

Now let $x\in V_T(S)$. We have 
\[
\|x-s\|\leq \|x-t\| \quad \forall s\in S, t\in T\setminus S.
\]
Let $\bar s\in S$ be such that $\|x-\bar s\|=\max_{s\in S}\|x-s\|$ (such $\bar s$ exists since we assumed that $S$ is closed). Then $S'' :=S\setminus \bar s \in \mathcal{S}_{K-1}$, and so 
\[
x\in V_{T}(S\setminus \bar s) =  V_{T}(S'') \subseteq  \bigcup_{S'\in \mathcal{S}_{K-1}}V_{T}(S').
\]

\end{proof}

It appears from empirical observations that the inclusions \eqref{eq:inclusions} are strict whenever $\interior V_T(S)\neq \emptyset$, however we weren't able to prove this. We discuss this in more detail in Section~\ref{sec:conclusions}.

\section{Neighbour Relations}\label{sec:neighbour}

Efficient construction of Voronoi cells relies on identifying the subset of the \emph{nearest neighbours} for this cell. If a cell $V_T(S)$ for some finite subsets $S$ and $T$ of $\R^n$ has a nonempty interior, then it may be natural to define the \emph{neighbours} of $S$ in $T$ as the set of all points $t$ in $T$ that define a facet of $V_T(S)$, i.e. such that the intersection  
$\{y\,|\, \|y-s\| = \|y-t\|\}\cap V_T(S)$ is a (convex polyhedral) set of dimension $n-1$. Even though this definition makes perfect sense for the classic Voronoi cells, it is unsuitable for higher-order cells, due to the existence of lower-dimensional cells that do not have facets of dimension $n-1$. 

We hence provide a more careful definition of neighbours that captures all special cases including empty and lower-dimensional cells. 

\begin{definition}
A subset $N\subseteq T\setminus S$ is a \emph{minimal} set of neighbours for the cell $V_T(S)$ if 
\[
V_T(S) = V_{N}(S)  \neq V_{N'}(S)  \qquad \forall\; N'\subsetneq N.
\]	
A site $t\in T$ is a \emph{neighbour} of $V_T(S)$ if it belongs to a minimal set of neighbours for $V_T(S)$. We denote the set of all neighbours of $S$ by $N_T(S)$.
\end{definition}

For some configurations of points it may be possible to choose several minimal  subsets of  neighbours that fully characterise the cell. The following example illustrates this point. 

\begin{example}\label{eg:essential} Consider a set $S$ that consists of two opposite points on the unit circle and the set $T$ that is a bunch of points on the circle on  the two semi-circles defined by the two points in $S$ (see Fig.~\ref{fig:ambiguous}).
\begin{figure}[ht]
{\centering
\includegraphics[width=0.4\textwidth]{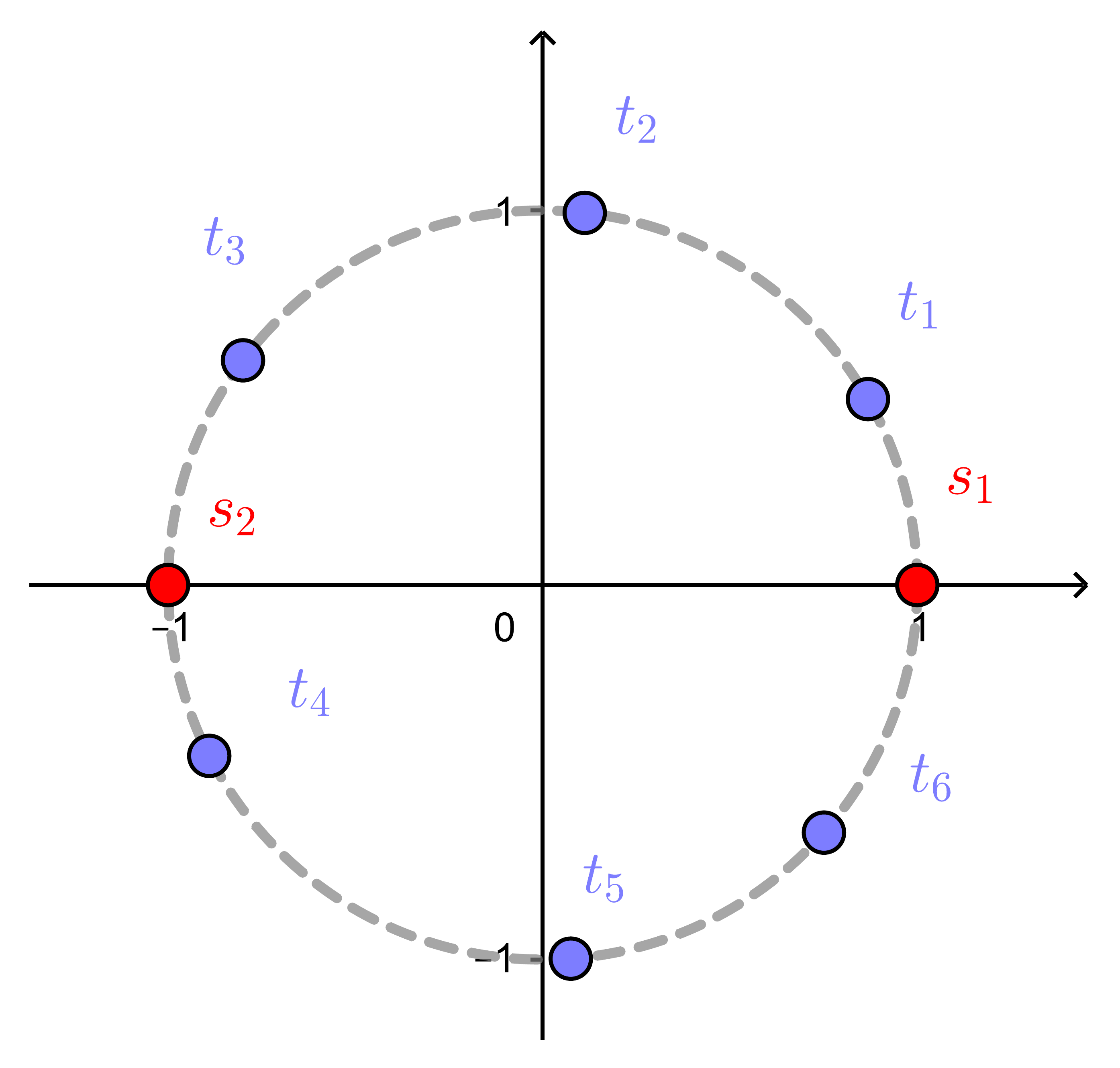}}
\caption{Several choices of the minimal set of neighbours}
\label{fig:ambiguous}
\end{figure}
It follows from Theorem~\ref{thm:ballchar} that the cell $V_T(S)$ is a singleton (centre of the circle) and removing any one of the points from $T\setminus S$ does not change the cell. It is hence possible to identify several different minimal subsets of neighbours that are sufficient to fully define the cell $V_T(S)$. 
\end{example}

Note that it may happen that for a pair of points  $(t,s) \in (T\setminus S)\times S$ their bisector has a nonempty intersection with the cell, however $t\notin N_T(S)$. The next example illustrates this observation.

\begin{example}\label{eg:nonn} Let $T = \{(-1,1),(-1,-1),(1,-1)\}$, $S = \{(1,1)\}$. Then (using  Proposition~\ref{prop:handycharacterisation})
\[
V_T(S) = \left\{  (x,y)\in\mathbb{R}^{2}\,:\, -2 x \leq 0, -2 y \leq 0\right\} = \{(x,y)\,|\, x,y\geq 0\}.
\]
In this case the only minimal set of neighbours is $N_T(S) = \{(-1,1),(1,-1)\}$, and while the bisector of the pair $((-1,-1),(1,1))$ touches the Voronoi cell $V_T(S)$ at the origin, the point $(-1,-1)$ is not a neighbour of this Voronoi cell (see Fig.~\ref{fig:nonneighbour}).
\begin{figure}[ht]
{\centering
\includegraphics[width=0.4\textwidth]{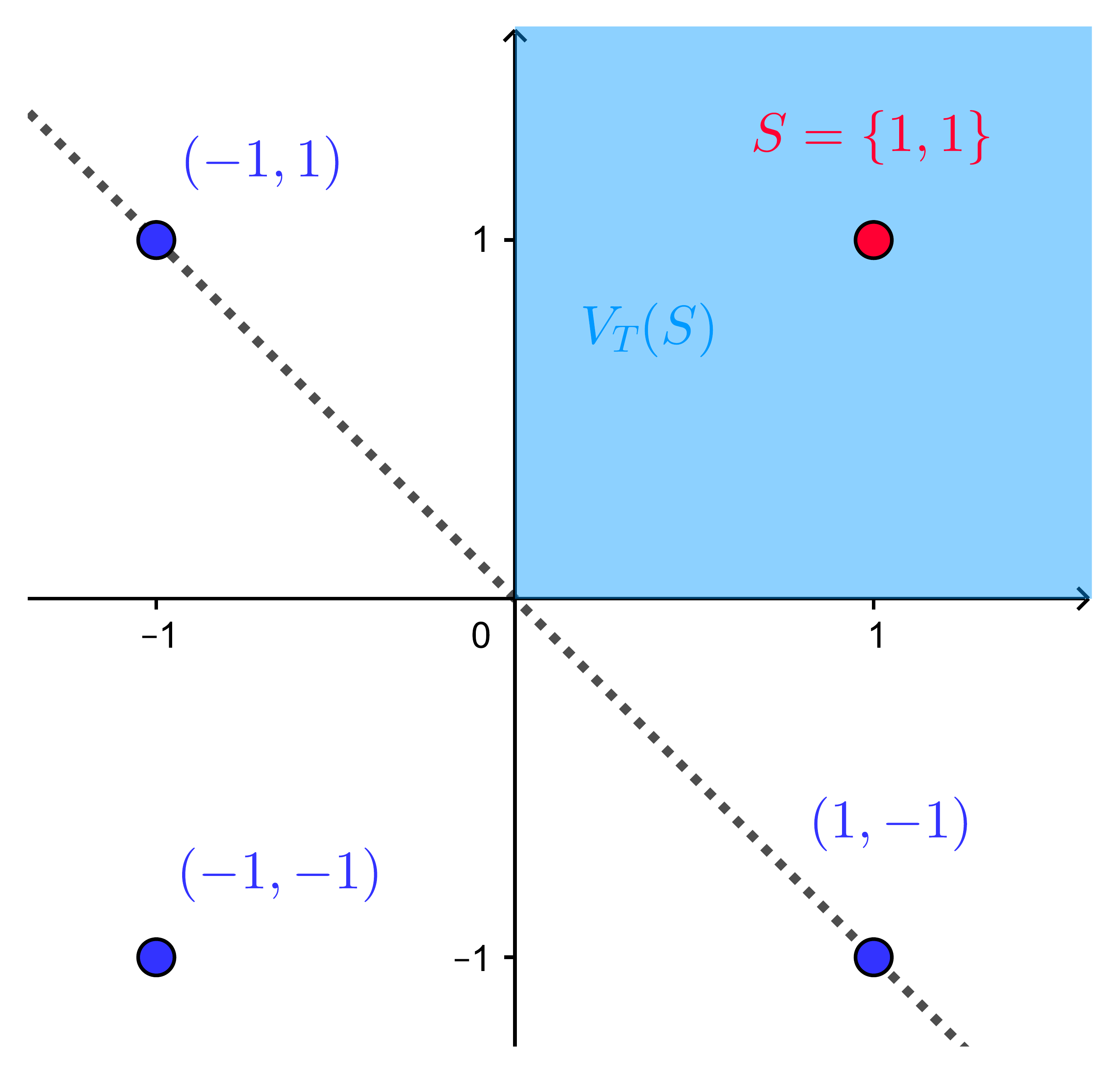}}
\caption{An example of a non-neighbour (see Example~\ref{eg:nonn}).}
\label{fig:nonneighbour}
\end{figure}
 
\end{example}

It is worth noting that if the set $T$ is discrete, and $S$ is compact, then the minimal set of neighbours for $V_T(S)$ is finite. Moreover, the minimal set of neighbours is unique for Voronoi cells with nonempty interiors, as we show next.

\begin{lemma}\label{lem:nintminimal} Suppose that $S,T\subset \R^n$, where $S$ is finite and $T$ is discrete, and assume that $S\subset T$. If $\interior V_T(S)\neq \emptyset$, then a minimal subset of neighbours $N$ is unique. Moreover, each $t\in N$ defines a facet of $V_T(S)$, i.e.
\[
\dim \left( \{y\,|\, \|y-s\| = \|y-t\|\}\cap V_T(S)\right) = n-1.
\]

\end{lemma}
\begin{proof} If $V_T(S)$ has a nonempty interior, then it is fully determined by the half-spaces defining its facets ($n-1$-dimensional faces). This is the minimal and uniquely defined set of half-spaces. Each one of these half-spaces is defined by a pair $(s,t)\in S\times (T\setminus S)$ via the inequality  $\|s-x\|\leq \|t - x\|$, which can be rewritten as
\[
\langle t-s, x\rangle \leq \frac{1}{2}\left(\|t\|^2 - \|s\|^2\right).
\]

Even though the minimal set of half-spaces is unique, we need to make sure that each half-space is defined by a unique pair $(s,t)\in S\times (T\setminus S)$. This follows from Proposition~2.6 in \cite{multipoint}, where it is proved that if two pairs $(s_1,t_1),(s_2,t_2)\in S\times (T\setminus S)$ define the same half-space, then the relevant inequality is nonessential for the cell, and hence can not define a facet. The last part of the lemma follows from the preceding discussion.
\end{proof}

\begin{lemma}[cf. Lemma~\ref{lem:inclusions}]\label{lem:neighbours} Let $T$ be a discrete subset of $\R^n$, and let $S$ be a proper subset of $T$ such that $\interior V_T(S)\neq \emptyset$. Denote by $\mathcal{S}_k$ the set of all subsets of $S$ of cardinality exactly $k$, for all $k\in \{1,\dots, K\}$, where $K=|S|$. Then
\begin{equation}\label{eq:relationsneigh}
N_{T}(S)  \subseteq \bigcup_{S'\in \mathcal{S}_{k-1}}N_{T\setminus S}(S')   \subseteq \bigcup_{S'\in \mathcal{S}_{k-2}} N_{T\setminus S}(S')  \subseteq \cdots \subseteq   \bigcup_{S'\in \mathcal{S}_{1}} N_{T\setminus S}(S')  .
\end{equation}
\end{lemma}
\begin{proof} It is sufficient to show that
\[
N_T(S)  \subseteq \bigcup_{S'\in \mathcal{S}_{|S|-1}}N_{T\setminus S}(S'),
\]
the rest follows by induction. To see that this is indeed true, recall that by Lemma~\ref{lem:nintminimal}  each point in the (unique) minimal set  of neighbours defines a facet of the cell $V_T(S)$. By Corollary~\ref{cor:intersection} we have 
\[
V_T(S) = \bigcap_{S'\in \mathcal{S}_{|S|-1}} V_{T\setminus S}(S'). 
\]
If for every $p\in N_T(S)$ there exists $ S' \in \mathcal{S}_{|S|-1}$ such that $p$ defines a facet of $V_{T\setminus S} (S')$, then we are done. Otherwise there is a $p\in N_{T}(S)$ that does not define facets of $|S|-1$-order cells. This means that we can choose an alternative minimal set of neighbours, a contradiction. 
\end{proof}

\begin{remark} Note that without the condition $\interior V_T(S)\neq \emptyset$ the statement of the preceding lemma may not be true: this is evident from Fig.~\ref{fig:ambiguous}.
\end{remark}

Finally, we prove a statement very similar to Lemma~\ref{lem:inclusions}, replacing $T\setminus S$ with $T$ in the chain of inclusions. Note that this may be more useful for practical purposes, as it allows to reduce the set of candidate neighbours from an existing lower-order diagram. 

\begin{lemma}\label{lem:neighboursT} Let $T$ be a discrete subset of $\R^n$, and let $S$ be a proper subset of $T$ such that $\interior V_T(S)\neq \emptyset$. As before, denote by $\mathcal{S}_k$ the set of all subsets of $S$ of cardinality exactly $k$, for all $k\in \{1,\dots, K\}$, where $K=|S|$. Then
\begin{equation}\label{eq:relationsneighT}
N_{T}(S)  \subseteq \bigcup_{S'\in \mathcal{S}_{k-1}}N_{T}(S')   \subseteq \bigcup_{S'\in \mathcal{S}_{k-2}} N_{T}(S')  \subseteq \cdots \subseteq   \bigcup_{S'\in \mathcal{S}_{1}} N_{T}(S')  .
\end{equation}
\end{lemma}
\begin{proof} Just as in the proof of Lemma~\ref{lem:inclusions}, it is sufficient to show that
\[
N_T(S)  \subseteq \bigcup_{S'\in \mathcal{S}_{|S|-1}}N_{T}(S').
\]
Suppose that this is not true, and for some configuration there exists $t\in N_T(S)$ such that 
\[
t\notin N_T(S')\quad \forall S'\in \mathcal{S}_{|S|-1}.
\]
Then
\[
V_T(S') = V_{T\setminus \{s\}}(S') \quad \forall S'\in \mathcal{S}_k, 
\]
and by Corollary~\ref{cor:intersection}
\[
V_T(S)  = \bigcap_{S'\in \mathcal{S}_k} V_{T}(S') = \bigcap_{S'\in \mathcal{S}_k} V_{T\setminus \{t\}}(S')  . 
\]
We deduce that $t$ is not in $N_T(S)$, a contradiction.
\end{proof}

\section{Conclusions}\label{sec:conclusions}

The motivation for this paper comes from the construction of regular tessellations of the Euclidean space via Voronoi diagrams on lattices on the one hand \cite{ryan}, and from several puzzling examples and negative results related to Voronoi diagrams on the plane obtained in \cite{multipoint}, on the other. We have succeeded in explaining the lack one-dimensional cells in $\R^2$ observed in  \cite{multipoint} proving that there can be no higher-order cells of dimension $n-1$ (see Theorem~\ref{thm:dimension}), however the question of generalising other negative results from that study remains open. For instance, it was shown in  \cite{multipoint} that in the case when $|T|=4$ and $S\subsetneq T$, $|S|\geq 2$, the higher-order cell $V_T(S)$ can not be a triangle or a quadrilateral. In contrast to line segments, which are impossible to realise as higher-order Voronoi cells in the plane, triangles and cyclic quadrilaterals can be realised using larger sets of sites, as shown in Fig.~\ref{fig:triangsquare}.
\begin{figure}[ht]
\includegraphics[width = 0.45\textwidth]{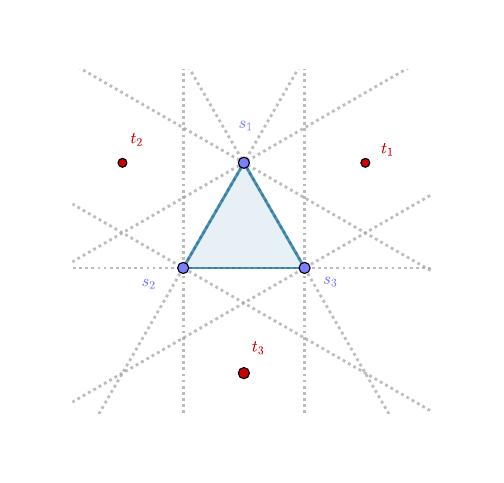}
\quad
\includegraphics[width = 0.45\textwidth]{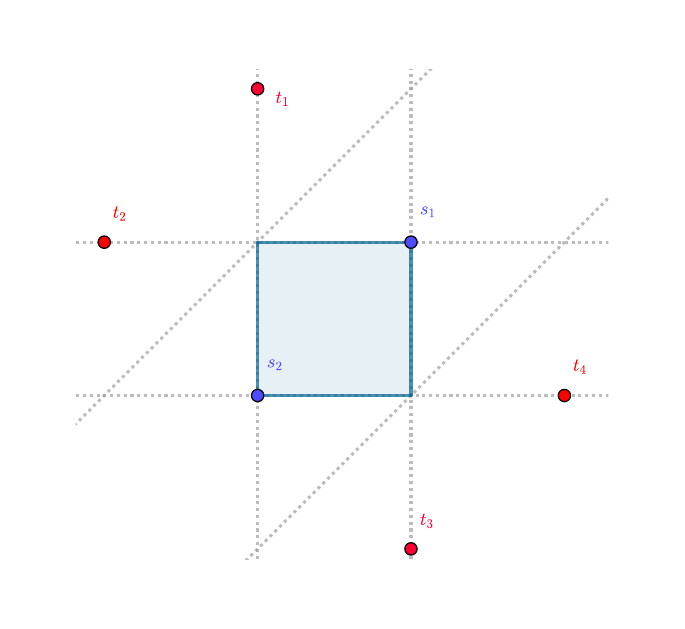}
\caption{Triangle and square as order 2 Voronoi cells}
\label{fig:triangsquare}
\end{figure}
The triangular cell $V_T(S)$ is the convex hull of the set $S = \{s_1,s_2,s_3\}$, and there are four inequalities intersecting at each vertex; likewise, for the square cell $V_T(S)$ the two sites of $S$ are opposite vertices of the square, and the construction requires four additional sites in $T$. There are redundant (coincident) inequalities at two vertices of the square, and hence this configuration appears non-generic. This brings us to the following question.

\begin{question}Is it true that higher-order Voronoi cells on discrete set of sites inscribed in a sphere exhibit degeneracy of sorts? Can this be captured in a rigorous way?
\end{question}

A related question is in the spirit of \cite{multipoint} and \cite{farthest} (where a general geometric characterisation of nonempty farthest cells was obtained): it is well-known that any polyhedral set with nonempty interior can be represented as a classic (first-order) Voronoi cell. On the other hand, we know from Corollary~\ref{cor:nodimn1} that cells of dimension $n-1$ are impossible in $\R^n$.

\begin{question} Is it true that any polyhedral set of dimension different to $n-1$ can be represented as a higher-order Voronoi cell in $\R^n$? What is the smallest number of sites $|T\cup S|$ needed to represent a given polytope as a $k$-th order Voronoi cell?
\end{question}

We have heavily utilised in our proofs the Euclidean ball characterisation from Theorem~\ref{thm:ballchar}. It appears that for the interior points of a Voronoi cell there is much flexibility in choosing the radius of the Euclidean ball that satisfies the theorem and certifies that the point belongs to the cell. However, for degenerate configurations (cells with empty interior and vertices) the choice of the radius appears restricted. 

\begin{question} Is there a relation between the radii of Euclidean balls satisfying \eqref{eq:ballchar} of Theorem~\ref{thm:ballchar} and the `singularity' or `degeneracy' of the centre point?
\end{question}

There are several minor questions that relate to the tightness of the results in this paper. These are not necessarily of paramount practical importance, however, answering these questions would enhance our understanding of the structure of higher-order Voronoi cells.

\begin{question}
Is it true that in nondegenerate cases (say when $\interior V_T(S)$ is nonempty) the inclusions \eqref{eq:inclusions} in Lemma~\ref{lem:inclusions} are strict?
\end{question}

\begin{question} Note that our main result Theorem~\ref{thm:dimension} is stated for the discrete set of sites. It would be interesting to figure out if this result is true in a more general setting.
\end{question}

\section{Acknowledgements}

The second author is grateful to the Australian Research Council for continuing support.

\bibliographystyle{plain}
\bibliography{references}
	
\end{document}